\documentclass[a4paper]{article}
\usepackage{amsmath}
\usepackage{amssymb}
\usepackage{amsthm}
\usepackage[english]{babel}
\usepackage[utf8x]{inputenc}
\usepackage[colorinlistoftodos]{todonotes}
\usepackage{esint}

\newtheorem{theorem}{\hspace*{\parindent}Theorem}
\newtheorem{lemma}{\hspace*{\parindent}Lemma}
\newtheorem{corollary}{\hspace*{\parindent}Corollary}
\newtheorem{prop}{\hspace*{\parindent}Proposition}

\def\a{\mathbf{a}}
\def\b{\mathbf{b}}

\def\R{\mathbb{R}}
\def\C{\mathbb{C}}
\def\N{\mathbb{N}}
\def\LP{\mathcal{L\!-\!P}}

\title{Log-concavity and Tur\'{a}n-type inequalities for the generalized  hypergeometric function}
\author{S.I.\:Kalmykov$^{\rm a}$ and D.B.\:Karp$^{\rm b}$\footnote{Corresponding author. E-mail: S.I.\:Kalmykov -- \emph{sergeykalmykov@inbox.ru}, D.B.\:Karp -- \emph{dimkrp@gmail.com}}
\\[10pt]\small{\textit{$\phantom{1}^a$Shanghai Jiao Tong University, Shanghai, China}}\\\small{\textit{$\phantom{1}^b$Far Eastern Federal University, Vladivostok, Russia}}}
\date{}
\begin{document}
\maketitle

\begin{abstract}
The paper studies logarithmic convexity and concavity of the generalized hypergeometric function with respect to simultaneous shift of several parameters. We use integral representations and properties of Meijer's $G$ function to prove log-convexity.  When all parameters are shifted we use series manipulations to examine the power series coefficients of the generalized Tur\'{a}nian formed by the generalized hypergeometric function.  In cases when all zeros of the generalized hypergeometric function are real, we further explore the consequences of the extended Laguerre inequalities and formulate a conjecture about reality of zeros.
\end{abstract}

\bigskip


\bigskip


\bigskip

\section{Introduction and preliminaries}

Throughout the paper we will use the standard definition of the generalized hypergeometric function ${_{p}F_q}$ as the sum of the series
\begin{equation}\label{eq:pFqdefined}
{_{p}F_q}\left(\left.\!\!\begin{array}{c}\a \\ \b\end{array}\right|z\!\right)={_{p}F_q}\left(\a;\b;z\right)
=\sum\limits_{n=0}^{\infty}\frac{(a_1)_n(a_2)_n\cdots(a_{p})_n}{(b_1)_n(b_2)_n\cdots(b_q)_nn!}z^n
\end{equation}
if $p\le{q}$, $z\in\C$ (the complex plane). If $p=q+1$ the above series only converges in the open unit disk and ${_{p}F_q}(z)$ is defined as its analytic continuation for $z\in\C\!\setminus\![1,\infty)$.  Here $(a)_n=\Gamma(a+n)/\Gamma(a)$ denotes the rising factorial (or Pochhammer's symbol) and $\a=(a_1,\ldots,a_p)$, $\b=(b_1,\ldots,b_q)$ are (generally complex) parameter vectors, such that $-b_j\notin\N_0$ (nonnegative integers), $j=1,\ldots,q$.  This last restriction can be easily removed by dividing both sides of (\ref{eq:pFqdefined}) by $\prod_{k=1}^{q}\Gamma(b_k)$.  The resulting function (known as the regularized generalized hypergeometric function) is entire in $\b$.
In what follows we will use the shorthand notation  for the products and sums:
\begin{equation*}
\begin{split}
&\Gamma(\a)=\Gamma(a_1)\Gamma(a_2)\cdots\Gamma(a_p),~~~(\a)_n=(a_1)_n(a_2)_n\cdots(a_p)_n,
\\
&\frac{(\a)}{(\b)}=\frac{(\a)_1}{(\b)_1}=\frac{a_1a_2\cdots a_p}{b_1b_2\cdots{b_q}},~~~\a+\mu=(a_1+\mu,a_2+\mu,\dots,a_p+\mu);
\end{split}
\end{equation*}
inequalities like $\a>0$ will be understood element-wise.

In a series of papers \cite{KK1,KK2,KK3,Karp11,Karp15,Karp16,KPJMAA,KSJMAA} we initiated an investigation of logarithmic convexity and concavity of the generalized hypergeometric function viewed as a function of parameters, as well as extensions to more general series containing hypergeometric terms. In particular, we found certain cases when the functions $\mu\to{f_i(\mu;x)}$ are log-concave/log-convex, where
\begin{multline*}
(f_1(\mu;x),f_2(\mu;x),f_3(\mu;x),f_4(\mu;x))
\\
=\biggl(1, \Gamma(\a_2+\mu),\frac{1}{\Gamma(\b_2+\mu)},  \frac{\Gamma(\a_2+\mu)}{\Gamma(\b_2+\mu)}\biggr)\times{_{p}F_q}\left(\left.\!\!\begin{array}{c}\a_1,\a_2+\mu\\\b_1,\b_2+\mu\end{array}\right|x\!\right).
\end{multline*}
Moreover, we studied the power series coefficients (in $x$) of the ''generalized Tur\'{a}nians''
\begin{equation}\label{eq:genTur}
\Delta_{f_i}(\alpha,\beta;x)=f_i(\mu+\alpha;x)f_i(\mu+\beta;x)-f_i(\mu;x)f_i(\mu+\alpha+\beta;x)
\end{equation}
under various restrictions on non-negative numbers $\alpha$ and $\beta$. A number of related results has also been established by several other authors in \cite{Baricz2008,BarIsm,BK,BPS,BGR}.  To the best of our knowledge, in all results obtained so far, with one exception, the vectors $\a_2$, $\b_2$ either contain one component or are empty. The single exception mentioned above is \cite[Theorem~6]{Karp15}, where the log-convexity of the function
\begin{equation}\label{eq:fmu-defined}
\mu\to f(\mu;x)=\frac{\Gamma(\a_2+\mu)}{\Gamma(\b_2+\mu)}{_{p}F_q}\left(\left.\!\!\begin{array}{c}\a_1,\a_2+\mu\\\b_1,\b_2+\mu\end{array}\right|x\!\right)
\end{equation}
is claimed for $\a_2$, $\b_2$ of arbitrary but equal length under certain additional restrictions.  The proof is only hinted to in \cite{Karp15} and, unfortunately, the multiplier $\Gamma(\a_2+\mu)/\Gamma(\b_2+\mu)$ is mistakenly missing in the formulation of \cite[Theorem~6]{Karp15}.   The first purpose of this paper is to give a complete proof of a strengthened and refined version of this theorem presented in the form of Theorems~\ref{th:master} and \ref{th:pFqpositive}. Log-convexity of $f(\mu)$ implies nonnegativity of $-\Delta_f$, the negative generalized Tur\'{a}nian, as defined in (\ref{eq:genTur}).   For $p_2=q_2$ we further complement this nonnegativity by establishing an upper bound in Theorem~\ref{th:master}.  In Theorem~\ref{th:pFqpositive} we elaborate on the conditions sufficient for the hypotheses of Theorem~\ref{th:master}.   Our second goal is to extend and complement Theorem~\ref{th:master} by considering the power series coefficients of this ''generalized Tur\'{a}nian'' for the particular case when $\a_1$ and $\b_1$ are empty. This is achieved in Theorem~\ref{th:pFp}, which is accompanied by two conjectures regarding its possible extensions. A consequence of this theorem is a log-concavity condition for the function $f(x)={_{p}F_q}(\a;\b;x)$. This log-concavity is equivalent to the inequality $[f'(x)]^2-f(x)f''(x)\ge0$ known as  the Laguerre inequality valid, in particular, for the entire functions in the Laguerre-P\'{o}lya class $\LP$. This class is defined as the set of real entire functions having the Hadamard factorization of the form
\begin{equation}\label{eq:LPclass}
f(x)=cx^ne^{-\alpha{x^2}+\beta{x}}\prod_{k=1}^{\infty}\left(1+\frac{x}{x_k}\right)e^{-\frac{x}{x_k}},
\end{equation}
where $c,\beta,x_k\in\R$ (the real line), $c\ne0$, $\alpha\ge0$, $n$ is a non-negative integer and $\sum_{k=1}^{\infty}1/x_k^2<\infty$.
Using an important observation due to Richards \cite{Rich} we conclude that for $\a,\b>0$ the generalized hypergeometric function ${}_pF_q$ belongs to $\LP$ when $p\le{q}$ and $a_k=b_k+n_k$ for $n_k\in\N_0$ and $k=1,\ldots,q$. Hence, under this additional restriction, the extended Laguerre inequalities due to Csordas, Varga \cite{CsV} and Patrick \cite{Pat} yield a sequence of inequalities for $x\to{}_pF_q(x)$ of which log-concavity is only  the first element.  This fact is presented in Corollary~\ref{cor:extendedLaguerre}.
Finally, we formulate a conjecture regarding the reality of zeros of ${}_pF_q(z)$.

\section{Logarithmic convexity}

Suppose $\a=(\a_1,\a_2)\in\R^{p}$  and  $\b=(\b_1,\b_2)\in\R^{q}$. We will write $p_1$, $p_2$ and $q_1$, $q_2$ for the dimensions of the subvectors $\a_1$, $\a_2$ and $\b_1$, $\b_2$, respectively.  We will always assume that  $\a_2$ is not empty (i.e. $p_2\ge1$), while all other subvectors are allowed to be empty. In this section we will consider log-convexity of the function $f(\mu;x)$ defined in (\ref{eq:fmu-defined}) under these assumptions.  The key role  will be played by the inequality
\begin{equation}\label{eq:v-defined}
v_{\a,\b}(t)=\sum\limits_{k=1}^{p}(t^{a_k}-t^{b_k})\ge0~\text{for}~t\in[0,1]
\end{equation}
for the \emph{M\"{u}ntz polynomial} $v_{\a,\b}(t)$ defined for two real vectors $\a$, $\b$ of equal size.   Inequality (\ref{eq:v-defined}) is implied by the stronger condition $\b\prec^W\a$ known as the weak supermajorization \cite[section~2]{KPCMFT}  and given by  \cite[Definition~A.2]{MOA}
\begin{equation}\label{eq:amajorb}
\begin{split}
& 0<a_1\leq{a_2}\leq\cdots\leq{a_p},~~
0<b_1\leq{b_2}\leq\cdots\leq{b_p},
\\
&\sum\limits_{i=1}^{k}a_i\leq\sum\limits_{i=1}^{k}b_i~~\text{for}~~k=1,2\ldots,p.
\end{split}
\end{equation}
Further sufficient conditions for (\ref{eq:v-defined}) in terms of $\a$, $\b$ can be found in our recent paper \cite[section~2]{KPCMFT}.
We will write $|\a|$ for the number of elements of $\a$ and $\a>0$ for $a_k>0$ for all $k$. First, we prove the following \emph{Master Theorem}.
\begin{theorem}\label{th:master}
Suppose $0\le{p_1}\le{q_1+1}$,  $p_2\ge1$, $0\le{q_2}\le{p_2}$, $p_1+p_2\leq{q_1+q_2+1}$,  $\a_2>0$ and there exists $\a_2'\subset\a_2$, $|\a_2'|=q_2$ such that $v_{\a_2',\b_2}(t)\ge0$ on $[0,1]$. Further, assume that for some $x\in\R$
\begin{equation}\label{eq:p1Fq1positive}
{}_{p_1}F_{q_1}\!\left(\!\!\!\left.\begin{array}{l}\a_1\\\b_1\end{array}\!\!\right|xt\!\right)\ge0~\text{for all}~t>0~\text{if}~p_1\le{q_1}~\text{or for}~0<t<1~\text{if}~p_1=q_1+1.
\end{equation}
Then for arbitrary $\alpha,\beta>0$ and $\mu\ge0$\emph{:}
\begin{equation}\label{eq:genTuranTh1}
0\le f(\mu;x)f(\mu+\alpha+\beta;x)-f(\mu+\alpha;x)f(\mu+\beta;x)\le\frac{1}{4}f^2(\mu;x),
\end{equation}
where $f(\mu;x)$ is defined in \emph{(\ref{eq:fmu-defined})} and the right hand inequality is true under additional assumption $p_2=q_2$.
The left hand inequality is equivalent to log-convexity of $\mu\to{f(\mu;x)}$ on $[0,\infty)$.
\end{theorem}
\begin{proof} Note first that for $p_1=q_1+1$ the inequality $p_1+p_2\leq{q_1+q_2+1}$ leads to the conclusion that $p_2=q_2$. In this case the condition $v_{\a_2,\b_2}(t)\ge0$ implies that $\sum_{i=p_1+1}^{p}(b_{i-1}-a_i)\ge0$ (this follows from $v'(1)\le0$ which is necessary since $v(1)=0$), so that all conditions of \cite[Theorem~1]{Karp15} are satisfied whether $p_2=q_2$ or $p_2>q_2$. Hence, if $\sum_{i=p_1+1}^{p}(b_{i-1}-a_i)>0$ representation \cite[(4)]{Karp15} takes the form $f(\mu;x)\!=\!\int_0^{\infty}p_x(t)dt$, where
$$
p_x(t)={}_{p_1}F_{q_1}\!\left(\!\!\!\left.\begin{array}{l}\a_1\\\b_1\end{array}\!\!\right|xt\!\right)t^{\mu-1}G^{p_2,0}_{q_2,p_2}\!\left(\!t~\vline\begin{array}{l}\b_2
\\\a_2\end{array}\!\!\right)
$$
and $G^{p_2,0}_{q_2,p_2}$ denotes Meijer's $G$ function (see \cite{KLMeijer1,KPJMAA} for its definition and basic properties).
Note, that $p_x(t)=0$ for $t>1$ if $p_2=q_2$ by \cite[Lemma~1]{KPJMAA}.  Using this notation the left hand inequality in (\ref{eq:genTuranTh1}) amounts to
\begin{equation*}
\int_0^{\infty}g(t)p_x(t)dt\int_0^{\infty}f(t)p_x(t)dt\le\int_0^{\infty}p_x(t)dt\int_0^{\infty}f(t)g(t)p_x(t)dt,
\end{equation*}
where $f(t)=t^{\beta}$ and $g(t)=t^{\alpha}$. Provided that $p_x(t)\ge0$ the required inequality is an instance of the Chebyshev inequality \cite[Chapter~IX (1.1)]{MPF}, since both $f(t)$ and $g(t)$ are increasing. The right hand inequality for $p_2=q_2$ follows from the weighted Gr\"{u}ss inequality \cite[(1.2)]{Dragomir}.  If $p_2=q_2$ and $\sum_{i=p_1+1}^{p}(b_{i-1}-a_i)=0$ inequality (\ref{eq:genTuranTh1}) follows from the previous case by continuity.

It is left to prove that $p_{x}(t)\ge0$. The first factor is nonnegative by the hypotheses of the theorem. If $p_2=q_2$, the $G$ function factor is nonnegative on $(0,1)$ by \cite[Theorem~2]{Karp15}. It is also non-negative (but could be infinite) at $t=1$ by left continuity.   For $p_2>q_2$ first recall that $\a_2>0$ by the hypotheses of the theorem. Combined with $v_{\a_2',\b_2}(t)\ge0$ this implies $\b_2>0$ as seen by examining $v_{\a_2',\b_2}(t)$ in the vicinity of $t=0$. Next, denote by $\tilde{\a}_2$ the subvector of $\a_2$ obtained by removing $\a_2'$. The sequence $\{(\a_2)_n/(\b_2)_n\}_{n\ge0}$ is the product of the Hausdorff moment sequence $\{(\a_2')_n/(\b_2)_n\}_{n\ge0}$ and the Stieltjes moment sequence $(\tilde{\a}_2)_n$ and so is itself a Stieltjes moment sequence.  Its representing measure is given by \cite[section~2]{KPCMFT}
$$
\frac{\Gamma(\b_2)}{\Gamma(\a_2)}\int\limits_{0}^{\infty}t^{n-1}G^{p_2,0}_{q_2,p_2}\!\left(\!t~\vline\begin{array}{l}\b_2\\\a_2\end{array}\!\!\right)dt
=\frac{(\a_2)_n}{(\b_2)_n}.
$$
This shows nonnegativity of the $G$ function factor for $p_2>q_2$.
\end{proof}

We will use the notation $\a_{[k]}$ for the vector $\a$ with $k$-th element removed, i.e. $\a_{[k]}=(a_1,\ldots,a_{k-1},a_{k+1},\ldots,a_p)$.
In our next theorem we list some cases when condition (\ref{eq:p1Fq1positive}) is satisfied.

\begin{theorem}\label{th:pFqpositive}
Inequality \emph{(\ref{eq:p1Fq1positive})} is true if any of the following conditions holds\emph{:}

\medskip

\emph{A)} $x\ge0$ and $0\le{p_1}<{q_1+1}$  or $0\le{x}<1$ and $1\le{p_1}=q_1+1$\emph{;}
$\{(\a_1)_n/(\b_1)_n\}_{n=0}^{\infty}$ is a positive sequence \emph{(}of course $\a_1,\b_1>0$ is sufficient but clearly not necessary for this to hold\emph{)}\emph{;}

\medskip

\emph{B)} $x<1$, $1\le{p_1}=q_1+1$, $\a_1'>0$ and $v_{\a_1',\b_1}(t)\ge0$ on $[0,1]$, where $\a_1'$ denotes $\a_1$ with the largest element removed\emph{;}

\medskip

\emph{C)} $x\in\R$, $p_1=q_1\ge0$, $\a_1>0$ and $v_{\a_1,\b_1}(t)\ge0$ on $[0,1]$\emph{;}

\medskip

\emph{D)} $x\in\R$, $1\le{p_1}=q_1-1$, $\a_1>0$ and $v_{\hat{\a}_{[k]},\hat{\b\!\!}_{\:\:[s]}}(t)\ge0$ on $[0,1]$, where $\hat{\a}=(\a_1,3/2)$, $\hat{\b\!\!}=\b_1$ and $\hat{a}_k\le\min\{1,\hat{b}_s-1\}$ for some indexes $k,s\in\{1,\ldots,q_1\}$.
\end{theorem}
\begin{proof} A) Indeed, under the stated hypotheses the function ${}_{p_1}F_{q_1}\!\left(\a_1;\b_1;xt\right)$ has positive power series coefficients and non-negative argument $xt$.

B)  Writing $\a_1=(\sigma,\a_1')$ we can conclude nonnegativity of ${}_{q_1+1}F_{q_1}\!\left(\a_1;\b_1;x\right)$ for all $x<1$ from the generalized Stieltjes transform representation \cite[(3)]{Karp15}. The $G$ function weight in this formula is nonnegative by \cite[Theorem~2]{Karp15}.

C) The function ${}_{q_1}F_{q_1}\!\left(\a_1;\b_1;x\right)$ is  nonnegative for all real $x$ by the Laplace transform representation \cite[(11)]{Karp15} with
nonnegative $G$ function weight. If $p_1=q_1=0$ the function ${}_{q_1}F_{q_1}\!\left(\a_1;\b_1;x\right)$ reduces to $e^x$.

D) These conditions imply nonnegativity of ${}_{q_1-1}F_{q_1}\!\left(\a_1;\b_1;x\right)$ for all real $x$ by \cite[Theorem~7]{KLMeijer1}.
\end{proof}

\textbf{Remark.} Conditions of Theorem~\ref{th:pFqpositive} when combined with the hypotheses of Theorem~\ref{th:master}  only leave three possibilities for $p$,$q$: $p=q+1$ (cases A and B), $p=q$ (cases A and C), $p=q-1$ (cases A and D).

\smallskip

Let us furnish some examples of how Theorems~\ref{th:pFqpositive} and \ref{th:master} can be used.

\smallskip

\textbf{Example~1}.  According to Theorem~\ref{th:pFqpositive}(A) the function
$$
\mu\to\frac{\Gamma(\a+\mu)}{\Gamma(\b+\mu)}{}_{q}F_{q}\!\left(\!\!\!\left.\begin{array}{l}\alpha_1,\alpha_2,\a+\mu\\\beta_1,\beta_2,\b+\mu\end{array}\!\!\right|x\!\right)
$$
is log-convex for $x>0$ if $\alpha_i,\beta_i<0$ ($-\beta_i\notin\N_0$) with $\lfloor\alpha_i\rfloor=\lfloor\beta_i\rfloor$,  $\a>0$ and $v_{\a,\b}(t)\ge0$ on $[0,1]$ (in particular if $\b\prec^W\a$). Here $\lfloor y\rfloor$ denotes the largest integer not exceeding $y$.

\textbf{Example~2}.  According to Theorem~\ref{th:pFqpositive}(B) with $p_1=q_1+1=1$ and $\a_1=\sigma$ (so that $\a_1'$ is empty vector), $p_2=q_2=q$ the function
$$
\mu\to\frac{\Gamma(\a+\mu)}{\Gamma(\b+\mu)}{}_{q+1}F_{q}\!\left(\!\!\!\left.\begin{array}{l}\sigma,\a+\mu\\\b+\mu\end{array}\!\!\right|x\!\right)
$$
is log-convex for arbitrary real $\sigma$, any $x<1$, $\a>0$ and $v_{\a,\b}(t)\ge0$ on $[0,1]$ (in particular if $\b\prec^W\a$).

\textbf{Example~3}. According Theorem~\ref{th:pFqpositive}(C) with $p_1=q_1=0$, the function
$$
\mu\to\frac{\Gamma(\a+\mu)}{\Gamma(\b+\mu)}{}_{q+1}F_{q}\!\left(\!\!\!\left.\begin{array}{l}\a+\mu\\\b+\mu\end{array}\!\!\right|x\!\right)
$$
is log-convex for all $x<1$ if $\a>0$ and $v_{\a',\b}(t)\ge0$ on $[0,1]$, where $\a'$ denotes $\a$ with one arbitrary element removed.
Similarly, by Theorem~\ref{th:pFqpositive}(C)
$$
\mu\to\frac{\Gamma(\a+\mu)}{\Gamma(\b+\mu)}{}_{q}F_{q}\!\left(\!\!\!\left.\begin{array}{l}\a+\mu\\\b+\mu\end{array}\!\!\right|x\!\right)
$$
is log-convex for all real $x$ if $\a>0$ and $v_{\a,\b}(t)\ge0$ on $[0,1]$.

\textbf{Example~4}. According to Theorem~\ref{th:pFqpositive}(D) with $p_1=q_1-1=1$, the function
$$
\mu\to\frac{\Gamma(\a+\mu)}{\Gamma(\b+\mu)}{}_{q-1}F_{q}\!\left(\!\!\!\left.\begin{array}{l}\alpha,\a+\mu\\\beta_1,\beta_2,\b+\mu\end{array}\!\!\right|x\!\right)
$$
is log-convex for all real $x$ if $0<\alpha\le1$, $\beta_1\ge\alpha+1$,  $\beta_2\ge3/2$, $\a>0$ and $v_{\a,\b}(t)\ge0$ on $[0,1]$.
It is log-convex for positive $x$ if $\alpha,\beta_1<0$ ($-\beta_1\notin\N_0$) with $\lfloor\alpha\rfloor=\lfloor\beta_1\rfloor$, $\beta_2\ge0$, $\a>0$  and $v_{\a,\b}(t)\ge0$ on $[0,1]$ by Theorem~\ref{th:pFqpositive}(A).  Of course, $\alpha,\beta_1\beta_2>0$ is also sufficient if $x>0$.

\section{Generalized Tur\'{a}nian and it Taylor coefficients}

In what follows we will assume that $\a_1$ and $\b_1$ are empty vectors and consider the generalized Tur\'{a}nian
\begin{equation}\label{eq:genTuran}
\Delta_{f}(\alpha,\beta;x):=f(\mu+\alpha;x)f(\mu+\beta;x)-f(\mu;x)f(\mu+\alpha+\beta;x)=\sum_{m=0}^{\infty}\delta_mx^m,
\end{equation}
where
\begin{equation}\label{eq:pFqshiftall}
f(\mu;x)=\frac{\Gamma(\a+\mu)}{\Gamma(\b+\mu)}{}_{p}F_{q}\!\left(\!\!\!\left.\begin{array}{l}\a+\mu\\\b+\mu\end{array}\!\!\right|x\!\right)
=\sum\limits_{n=0}^{\infty}\frac{\Gamma(\a+\mu+n)}{\Gamma(\b+\mu+n)}\frac{x^n}{n!}.
\end{equation}
We will be interested not only in the sign of $\Delta_{f}(\alpha,\beta;x)$ but also in the sign of its power series coefficients $\delta_m$. Set $\mathbb{R}_{+}=[0,\infty)$.  The next two lemmas are found in \cite[Lemmas~2,3]{KK2}.
\begin{lemma}\label{l:dWlc}
Let $f$ be any function $\mathbb{R}_{+}\to\mathbb{R}_{+}$ and suppose that the generalized Tur\'{a}nian
$$
\Delta_{f}(\alpha,\beta)=f(\mu+\alpha)f(\mu+\beta)-f(\mu)f(\mu+\alpha+\beta)
$$
is nonnegative \emph{(}non-positive\emph{)} for $\alpha=1$ and all $\mu,\beta\ge0$. Then $\Delta_{f}(\alpha,\beta)\ge0$ $(\le0)$ for all $\alpha\in\mathbb{N}$ and $\mu,\beta\ge0$.
The inequality in conclusion is strict if so is the inequality in the hypotheses.
\end{lemma}

\begin{lemma}\label{l:coeff}
Let $f$ be defined by the series
$$
f(\mu;x)=\sum_{k=0}^{\infty}f_k(\mu)x^k,~~\text{where}~f_k(\mu)~\text{are arbitrary functions}, 
$$
and suppose $\Delta_{f}(1,\beta;x)$ defined in \emph{(\ref{eq:genTuran})} has nonnegative \emph{(}non-positive\emph{)} coefficients at all powers of $x$ for all $\mu,\beta\ge0$. Then $\Delta_{f}(\alpha,\beta;x)$ has nonnegative \emph{(}non-positive\emph{)} coefficients at powers of $x$ for all $\alpha\in\mathbb{N}$, $\alpha\leq\beta+1$ and $\mu\ge0$.
\end{lemma}

Next consider the rational function
\begin{equation}\label{eq:Rpq}
R_{p,q}(x)=\frac{\prod_{k=1}^{p}(a_k+x)}{\prod_{k=1}^{q}(b_k+x)}
\end{equation}
with  positive $a_k,b_k$. Let $e_m(\mathbf{c})=e_m(c_1,\dots,c_q)$ denote the $m$-th elementary symmetric polynomial,
$$
e_0(c_1,\dots,c_q)=1,~~~~~e_1(c_1,\dots,c_q)=c_1+c_2+\dots+c_q,
$$
$$
e_2(c_1,\dots,c_q)=c_1c_2+c_1c_3+\dots+c_1c_q+c_2c_3+\dots+c_2c_q+\dots+c_{q-1}c_{q},\ldots,
$$
$$
e_q(c_1,\dots,c_q)=c_1c_2\cdots{c_q}.
$$
We will need the following lemma. It is an extended version of \cite[Lemma~2]{KSJMAA} and its proof repeats \emph{mutatis mutandis} the corresponding result in \cite{KSJMAA}.

\begin{lemma}\label{lem:incrdecr}
If $p\geq q$ and
\begin{equation}\label{eq:incr}
\frac{e_p(\a)}{e_q(\b)}\leq \frac{e_{p-1}(\a)}{e_{q-1}(\b)}\leq\dots \leq \frac{e_{p-q+1}(\a)}{e_1(\b)}\leq e_{p-q}(\a),
\end{equation}
then the function $R_{p,q}(x)$ is monotone increasing on $(0,\infty)$.

If $p\leq q$ and
\begin{equation}\label{eq:decr}
\frac{e_q(\b)}{e_p(\a)}\leq \frac{e_{q-1}(\b)}{e_{p-1}(\a)}\leq \dots \leq \frac{e_{q-p+1}(\b)}{e_1(\a)}\leq e_{q-p}(\b),
\end{equation}
then the function $R_{p,q}(x)$ is monotone decreasing on $(0,\infty)$.
\end{lemma}

\begin{theorem} \label{th:pFp}
If $p\le{q}$ and conditions $(\ref{eq:decr})$ are satisfied, then $\Delta_f(\alpha,\beta;x)\ge0$ for $x\ge0$, $\mu,\beta\ge0$ and $\alpha\in\N$, where
$\Delta_f(\alpha,\beta;x)$ is defined in $(\ref{eq:genTuran})$ with $f$ from $(\ref{eq:pFqshiftall})$.  Moreover, if $\alpha\le\beta+1$, then $\delta_m\ge0$ for all $m\in\N_0$.

If $p\ge{q}$  and conditions $(\ref{eq:incr})$ are satisfied, then $\Delta_f(\alpha,\beta;x)\le0$ for $x\geq 0$, $\beta\ge0$ and $\alpha\in\N$. Moreover, if
$\alpha\le\beta+1$, then $\delta_m\le0$ for all $m\in\N_0$.
\end{theorem}

\begin{proof}
According to Lemmas~\ref{l:dWlc} and \ref{l:coeff} applied to the function $f$ defined in (\ref{eq:pFqshiftall}), it suffices to consider $\Delta_f(1,\beta;x)$.
Straightforward calculation  yields:
\begin{multline*}
\frac{\Gamma(\b+\mu)\Gamma(\b+\mu+\beta)}{\Gamma(\a+\mu)\Gamma(\a+\mu+\beta)}\Delta_f(1,\beta;x)=
\frac{(\a+\mu)}{(\b+\mu)} {}_pF_{q}\left(\left.\begin{array}{l}\a+\mu+1\\\b+\mu+1\end{array}\right|x\right)
{}_pF_{q}\left(\left.\begin{array}{l}\a+\mu+\beta\\\b+\mu+\beta \end{array}\right|x\right)
\\
-\frac{(\a+\mu+\beta)}{(\b+\mu+\beta)}{}_pF_{q}\left(\left.\begin{array}{l}\a+\mu\\\b+\mu\end{array}\right|x\right)
 {}_pF_{q}\left(\left.\begin{array}{l}\a+\mu+\beta+1\\\b+\mu+\beta+1\end{array}\right|x\right)
\\
={}_pF_{q}\left(\left.\begin{array}{l}\a+\mu+\beta\\\b+\mu+\beta\end{array}\right|x\right)
\frac{d}{dx}{}_pF_{q}\left(\left.\begin{array}{l}\a+\mu\\\b+\mu\end{array}\right|x\right)
 -{}_pF_{q}\left(\left.\begin{array}{l}\a+\mu\\\b+\mu\end{array}\right|x\right)
 \frac{d}{dx}{}_pF_{q}\left(\left.\begin{array}{l}\a+\mu+\beta\\\b+\mu+\beta \end{array}\right|x\right)
\\
=\sum\limits_{j=1}^{\infty}\frac{(\a+\mu)_j j }{(\b+\mu)_j }\frac{x^{j-1}}{j!} \sum\limits_{j=0}^{\infty} \frac{(\a+\mu+\beta)_j}{(\b+\mu+\beta)_j}\frac{x^j}{j!}   - \sum\limits_{j=0}^{\infty} \frac{(\a+\mu)_j}{(\b+\mu)_j}\frac{x^j}{j!}    \sum\limits_{j=1}^{\infty}\frac{(\a+\mu+\beta)_j j }{(\b+\mu+\beta)_j }\frac{x^{j-1}}{j!}
\\
=\sum\limits_{m=1}^{\infty}x^{m-1}\sum\limits_{k=0}^{m}\frac{(\a+\mu)_{k}k(\a+\mu+\beta)_{m-k}}{(\b+\mu)_{k}(\b+\mu+\beta)_{m-k}k!(m-k)!} - \sum\limits_{m=1}^{\infty}x^{m-1}\sum\limits_{k=0}^{m}\frac{(\a+\mu)_k(\a+\mu+\beta)_{m-k}(m-k)}{(\b+\mu)_{k}(\b+\mu+\beta)_{m-k}k!(m-k)!}
\\
=\sum\limits_{m=1}^{\infty}x^{m-1}\sum\limits_{k=0}^{m}\frac{(\a+\mu)_k(\a+\mu+\beta)_{m-k}}{(\b+\mu)_k(\b+\mu+\beta)_{m-k}k!(m-k)!}(2k-m)
\\
=\sum\limits_{m=1}^{\infty}\frac{x^{m-1}}{m!}\sum\limits_{0\leq k\leq m/2} \binom{m}{k}(m-2k)\left[\frac{(\a+\mu+\beta)_k(\a+\mu)_{m-k}}{(\b+\mu+\beta)_k(\b+\mu)_{m-k}}-\frac{(\a+\mu)_{k}(\a+\mu+\beta)_{m-k}}{(\b+\mu)_{k}(\b+\mu+\beta)_{m-k}}\right],
\end{multline*}
where we have made use of the well known and easily verifiable identity  \cite[p.405, formula (16.3.1)]{OLBC}
\begin{equation}\label{eq:pFqderivative}
\frac{d}{dx}{}_pF_{q}\left(\left.\begin{array}{l}\a\\\b\end{array}\right|x\right)=\frac{(\a)}{(\b)}{}_pF_{q}\left(\left.\begin{array}{l}\a+1\\\b+1\end{array}\right|x\right).
\end{equation}
The last equality is obtained by the Gauss pairing and in view of the fact that the unpaired  middle term vanishes due to the factor $(2k-m)$.
Finally, for $k\le{m-k}$ we can factor the term in brackets as follows:
\begin{multline*}
\frac{(\a+\mu+\beta)_k(\a+\mu)_{m-k}}{(\b+\mu+\beta)_k(\b+\mu)_{m-k}}-\frac{(\a+\mu)_k(\a+\mu+\beta)_{m-k}}{(\b+\mu)_k(\b+\mu+\beta)_{m-k}}\\
=\frac{(\a+\mu)_k(\a+\mu+\beta)_k}{(\b+\mu)_k(\b+\mu+\beta)_k} \left\{\frac{(\a+\mu+k)\dots(\a+\mu+m-k-1)}{(\b+\mu+k)\dots(\b+\mu+m-k-1)}\right. \\-\left.\frac{(\a+\mu+\beta+k)\dots(\a+\mu+\beta+m-k-1)}{(\b+\mu+\beta+k)\dots(\b+\mu+\beta+m-k-1)}\right\}
\\
=\frac{(\a+\mu)_k(\a+\mu+\beta)_k}{(\b+\mu)_k(\b+\mu+\beta)_k}
\left\{\prod\nolimits_{j=k}^{m-k-1}R_{p,q}(\mu+j)-\prod\nolimits_{j=k}^{m-k-1}R_{p,q}(\mu+\beta+j)\right\},
\end{multline*}
where $R_{p,q}(x)$ is defined in (\ref{eq:Rpq}). The theorem now follows from Lemma~\ref{lem:incrdecr} since $\mu,\beta\ge0$.
\end{proof}

The Tur\'{a}nian $\Delta_f(\alpha,\beta;x)$ is symmetric in $\alpha$, $\beta$ while the conditions of Theorem~\ref{th:pFp} are not.
Of course we can exchange the roles of $\alpha$ and $\beta$ and require that $\beta\in\N$ and $\beta\leq\alpha+1$. Note that in either case $\max(\alpha,\beta)\ge1$.
On the other hand, we found numerical counterexamples to Theorem~\ref{th:pFp} when $0\le\mu,\alpha,\beta<1$.  This argument and numerical experiments motivate the following two conjectures.

\medskip

\textbf{Conjecture~1}.  All conclusions of Theorem~\ref{th:pFp} hold for all $\alpha,\beta\ge0$  if $\mu\ge1$.

\medskip

\textbf{Conjecture~2}.  Suppose $p_1\le{q_1}$, $p_2\le{q_2}$ and conditions (\ref{eq:decr}) hold for $\a_1$, $\b_1$ and $\a_2$, $\b_2$. Then
the function $\mu\to {f(\mu;x)}$ defined in (\ref{eq:fmu-defined}) is log-concave on $[1,\infty)$ and the corresponding generalized Tur\'{a}nian $\Delta_f(\alpha,\beta;x)$ has nonnegative coefficients at all powers of $x$.

Example~3 from the previous section shows that if $v_{\a',\b}\ge0$ on $[0,1]$ the inequality $\Delta_f(\alpha,\beta;x)\le0$ holds for all $\mu,\alpha,\beta\ge0$ and $x<1$ (if $p=q+1$) or $x\in\R$ (if $p=q$). This shows indirectly that condition $v_{\a',\b}\ge0$ on $[0,1]$ is stronger that (\ref{eq:incr}). It is probably hard to prove this directly.  However, we know that the majorization condition $\b\prec^W\a$ defined in (\ref{eq:amajorb}) is not only sufficient for (\ref{eq:v-defined}) but is also known to be necessary when $p=2$ and not too far from being necessary in general.  The good news is that this condition admits a clear comparison with (\ref{eq:incr}).  For $p=q$ this comparison was made in \cite[Lemma~2]{Karp13}. For general $p$ and $q$ we get the following lemma.

\begin{lemma}\label{lm:major-elementary}
Let $p\ge{q}$, $\a\in\R^{p}$, $\b\in\R^{q}$ be positive vectors  and suppose that there exists $\a'\subset\a$, $|\a'|=q$ such that $\b\prec^W\a'$.  Then inequalities \emph{(\ref{eq:incr})} hold.  Similarly, inequalities \emph{(\ref{eq:decr})} hold if $p\le{q}$ and $\a\prec^W\b'$, where $\b'$ stands for some subset of $\b$ containing $p$ elements.
\end{lemma}
\begin{proof} We can assume without loss of the generality that $\a'=(a_{p-q+1},\ldots,a_p)$.  We will also write $\a_{[1]}$ for $(a_{2},\ldots,a_p)$.
Put $p-q=k$. We will prove the lemma by induction in $k$.  For $k=0$ the result is given in \cite[Lemma~2]{Karp13}.  Suppose it holds for $k-1$, so that  $e_{j-1+k}(\a_{[1]})/e_{j}(\b)\le{e_{j-2+k}(\a_{[1]})/e_{j-1}(\b)}$ for $j=1,\ldots,q$.  We need to show that
$$
\frac{e_{q+k}(\a)}{e_{q}(\b)}\le\frac{e_{q+k-1}(\a)}{e_{q-1}(\b)}\le\cdots\le\frac{e_{k+1}(\a)}{e_{1}(\b)}\le e_{k}(\a).
$$
Using basic properties of elementary symmetric polynomials this amounts to
$$
\frac{a_1e_{q+k-1}(\a_{[1]})}{e_{q}(\b)}\le\frac{a_1e_{q+k-2}(\a_{[1]})+e_{q+k-1}(\a_{[1]})}{e_{q-1}(\b)}\le\cdots\le\frac{a_1e_{k}(\a_{[1]})+e_{k+1}(\a_{[1]})}{e_{1}(\b)}\le a_1e_{k-1}(\a_{[1]})+e_{k}(\a_{[1]}).
$$
Taking $j=q$ in the induction hypothesis, we immediately get the leftmost inequality above. The remaining inequalities have the form
$$
\frac{a_1e_{j-1+k}(\a_{[1]})+e_{j+k}(\a_{[1]})}{e_{j}(\b)}\le\frac{a_1e_{j-2+k}(\a_{[1]})+e_{j-1+k}(\a_{[1]})}{e_{j-1}(\b)}
$$
for $j=1,\ldots,q-1$.  Dividing each numerator term by the corresponding denominator on both sides we see that the first terms satisfy the required inequality by the induction hypothesis. It remains to show that
$$
\frac{e_{j+k}(\a_{[1]})}{e_{j}(\b)}\le\frac{e_{j-1+k}(\a_{[1]})}{e_{j-1}(\b)}~\Leftrightarrow~\frac{e_{j+k}(\a_{[1]})}{e_{j-1+k}(\a_{[1]})}\le\frac{e_{j}(\b)}{e_{j-1}(\b)}
$$
for $j=1,\ldots,q-1$.  The last inequality is proved by combining Newton's inequalities with the induction hypothesis:
$$
\frac{e_{j+k}(\a_{[1]})}{e_{j-1+k}(\a_{[1]})}\leq\frac{e_{j-1+k}(\a_{[1]})}{e_{j-2+k}(\a_{[1]})}\le\frac{e_{j}(\b)}{e_{j-1}(\b)}.
$$
The second claim follows by exchanging the roles of $\a$ and $\b$.
\end{proof}

\begin{corollary}\label{cor:lcx}
If $p\le{q}$ and conditions $(\ref{eq:decr})$ are satisfied, then the function
$x\to{}_pF_{q}(x)$ is log-concave on $(0,\infty)$ which is equivalent to the Laguerre inequality
\begin{equation}\label{eq:difdif}
{}_pF_{p}'\left(\left.\begin{array}{l}\a\\ \b\end{array}\right|x\right)^2-{}_pF_{p}\left(\left.\begin{array}{l}\a\\ \b\end{array}\right|x\right){}_pF_{p}''\left(\left.\begin{array}{l}\a\\ \b\end{array}\right|x\right)\geq 0.
\end{equation}
If $p\ge{q}$  and conditions $(\ref{eq:incr})$ are satisfied,  then the function
$x\rightarrow {}_pF_{q}(x)$ is log-convex on $(0,\infty)$ and inequality \emph{(\ref{eq:difdif})} is reversed.
\end{corollary}
\begin{proof} Setting $\mu=0$ and $\alpha=\beta=1$, applying the derivative formula (\ref{eq:pFqderivative})
and some equalities from the chain in the proof of Theorem~\ref{th:pFp} we get:
$$
\frac{\Gamma(\b)^2}{\Gamma(\a)^2}\Delta_f(1,1;x)
={}_pF_{q}'\left(\left.\begin{array}{l}\a\\\b\end{array}\right|x\right)^2
 -{}_pF_{q}\left(\left.\begin{array}{l}\a\\\b\end{array}\right|x\right)
{}_pF_{q}''\left(\left.\begin{array}{l}\a\\\b\end{array}\right|x\right).
$$
The claims now follow directly from Theorem~\ref{th:pFp}.
\end{proof}

The Laguerre inequality $(f')^2-ff''\ge0$ is known to hold on the whole real line for functions $f$ from the Laguerre-P\'{o}lya class $\LP$ defined by the Hadamard factorization (\ref{eq:LPclass}) given in the introduction. Finding conditions of parameters ensuring that ${}_pF_q\in\LP$ is, in general, an interesting open problem.  However, some partial results are known which we present in the form of the next theorem.

\begin{theorem}\label{th:pFqinLP}
Suppose $p\le{q}$, $\a,\b>0$ and $\a$ can be re-indexed so that $a_k=b_k+n_k$ for $n_k\in\N_0$ and $k=1,\ldots,p$. Then
\begin{equation}\label{eq:phi}
\phi(z)={}_pF_q\left(\left.\begin{array}{l}\a\\ \b\end{array}\right|z\right)
=e^{az}\prod_{k=1}^{\omega} \left(1+\frac{z}{z_k}\right)e^{-\frac{z}{z_k}}\in\LP,
\end{equation}
where $z_k>0$, $\omega\le\infty$ and the series $\sum_{n=1}^{\infty}1/z^2_n$ converges.  Furthermore, if $p=q$, $\a\in\R$ contains no non-positive integers and $\b>0$ then  $a_k=b_k+n_k$ for $n_k\in\N_0$ is necessary and sufficient for $\phi\in\LP$.
\end{theorem}
\begin{proof}
Richards in \cite[pp.477-478]{Rich} observed that for $p\le{q}$, $a_k=b_k+n_k$ for $n_k\in\N_0$ and $b_k>0$, $k=1,\ldots,q$, the function $\phi(z)$ has only negative real zeros and genus $0$ or $1$ depending on whether $p<q$ or $p=q$, respectively. Moreover, Ki and Kim \cite[Theorem~3]{KiKim} showed that if $p=q$, then the function $\phi(z)$ has only real zeros (and their number is finite) if and only if $\a$ can be re-indexed so that $a_k=b_k+n_k$ for $n_k\in\N_0$ and $k=1,\ldots,p$. Now the claim follows by Hadamard's factorization theorem (see \cite[p. 26]{Lev} or \cite[p.250]{Titch}).
\end{proof}

\textbf{Remark.} Richards' result has also been instrumental in discovering a number of properties of hypergeometric polynomials in \cite{DJM}.

The immediate corollary of the above result is that under conditions of Theorem~\ref{th:pFqinLP} inequality (\ref{eq:difdif}) holds for all real $x$, so that $x\to{}_pF_{q}(x)$ is log-concave on the whole real line.  In fact, more can be said on employing the next proposition due to Patrick \cite{Pat} and Csordas and Varga \cite{CsV}.
\begin{prop}\label{prop:PCV}
Let
$$
f(z)=e^{-bz^2}f_1(z), \ \ (b\geq 0, f(z)\not\equiv 0),
$$
where $f_1(z)$ is a real entire function of genus $0$ or $1$. Set
$$
L_n[f](x)=\sum_{k=0}^{2n} \frac{(-1)^{k+n}}{(2n)!}\binom{2n}{k}f^{(k)}(x)f^{(2n-k)}(x)
$$
for $x\in \mathbb{R}$ and $n\geq 0$. Then $f(z)\in \mathcal{L\!-\!P}$ if and only if
$$
L_n[f](x)\ge0
$$
for all $x\in\R$ and $n\geq0$.
\end{prop}
Note that $L_1[f](x)\ge0$ is the classical Laguerre inequality. Combining Theorem~\ref{th:pFqinLP} with Proposition~\ref{prop:PCV} we are immediately led to
\begin{corollary}\label{cor:extendedLaguerre}
Under hypotheses of Theorem~\ref{th:pFqinLP} the inequalities
\begin{equation}\label{eq:Lnf}
L_n[\phi](x)\ge0,
\end{equation}
where $\phi$ is defined in \emph{(\ref{eq:phi})}, hold for all integer $n\ge0$ and all $x\in\R$.
\end{corollary}

By Proposition~\ref{prop:PCV} the inequalities $L_n[\phi](x)\ge0$ are necessary and sufficient for reality of all zeros of $\phi$.  These inequalities can be used for effective numerical verification of conjectures regarding conditions on parameters of ${}_pF_q$ that guarantee the reality of all its zeros.  We used this approach to carry out  extensive numerical testing of the following

\textbf{Conjecture~3}.  Suppose $p<q$, $\b>0$ and $a_k>b_k$ for $k=1,\ldots,p$. Then all zeros of ${}_pF_q(\a;\b;z)$ are real and negative.

\paragraph{Acknowledgements.} The research of the second author has been supported by the Ministry of Education of the Russian Federation (project 1398.2014) and by the Russian Foundation for Basic Research (project 15-56-53032).

\end{document}